\documentclass[pdflatex,sn-mathphys-num]{sn-jnl}

\usepackage{graphicx}%
\usepackage{multirow}%
\usepackage{amsmath,amssymb,amsfonts}%
\usepackage{amsthm}%
\usepackage{mathrsfs}%
\usepackage[title]{appendix}%
\usepackage{xcolor}%
\usepackage{textcomp}%
\usepackage{manyfoot}%
\usepackage{booktabs}%
\usepackage{algorithm}%
\usepackage{algorithmicx}%
\usepackage{algpseudocode}%
\usepackage{listings}%

\global\long\def\Re{\operatorname{Re}}

\global\long\def\Im{\operatorname{Im}}

\global\long\def\Arg{\operatorname{Arg}}

\theoremstyle{thmstyleone}%
\newtheorem{theorem}{Theorem}
\theoremstyle{thmstyletwo}%

\theoremstyle{thmstylethree}%
\newtheorem{lemma}[theorem]{Lemma}

\begin{document}

\title[Sum of Gegenbauer polynomials]{On zero distribution of the sum of Gegenbauer polynomials}


\author*[1]{\fnm{Khang} \sur{Tran}}\email{khangt@mail.fresnostate.edu}

\author[1]{\fnm{Riley} \sur{White}}\email{rcwhite@mail.fresnostate.edu}
\equalcont{These authors contributed equally to this work.}

\affil*[1]{\orgdiv{Department of Mathematics}, \orgname{California State University, Fresno}, \orgaddress{\street{5245 North Backer Avenue M/S PB108}, \city{Fresno}, \postcode{93740}, \state{CA}, \country{USA}}}


\abstract{We study the zero distribution of the sum of the first $n$ Gegenbauer polynomials
\[
S_n^{(\alpha)}(z)=\sum_{k=0}^n C_{k}^{(\alpha)}(z).
\]
We prove that for $\alpha>1$, the zeros of $S_n^{(\alpha)}(z)$ lie on the interval $(-1,1)$ for all large $n$.}

\keywords{Zero distribution; Asymptotics; Generating Function; Gegenbauer polynomials}


\pacs[MSC Classification]{30C15, 26C10, 11C08}

\maketitle

\section{Introduction}
The sequence of Gegenbauer polynomials, $\left\{ C_n^{(\alpha)}(z)\right\}_{n=0}^\infty$, generated by 
\begin{equation} \label{eq:Gegenbauergen}
\sum_{n=0}^\infty C^{(\alpha)}_n(z)t^n=\frac{1}{(1-2zt+t^2)^\alpha}
\end{equation}
play an important role in mathematics. These polynomials generalize many other polynomials, such as the Chebyshev polynomials of the second kind ($\alpha=1$) and Legendre polynomials ($\alpha=1/2$). For each $n\in \mathbb{N}$, the sum of the first $n$ Chebyshev polynomials of the first kind is important since it relates to the famous Dirichlet's kernel,
\[
D_n(x)= \frac{\sin((n+1/2)x)}{\sin (x/2)}.
\]
From this relation, it is known that the zeros of the sum of the first $n$ Chebyshev polynomials of the first and second kind lie on the interval $(-1,1)$  \cite{rivlin} (or \cite{at} for a generalized result). However, the zeros of the sum of the first $n$ Legendre polynomials, $\sum_{k=0}^n L_n(z)$, are not all real (see Figure \ref{fig:sumLegendre} for the zeros of this sum for $n=20$). Given the difference between the zero distribution of the sum of Chebyshev polynomials and that of Legendre polynomials, it is natural to ask for conditions on $\alpha$ under which the zeros of the sum of Gegenbauer polynomials, $\sum_{k=0}^n C_k^{(\alpha)}(z)$, lie on $(-1,1)$.   
\begin{figure} 
    \centering
    \includegraphics[width=0.5\linewidth]{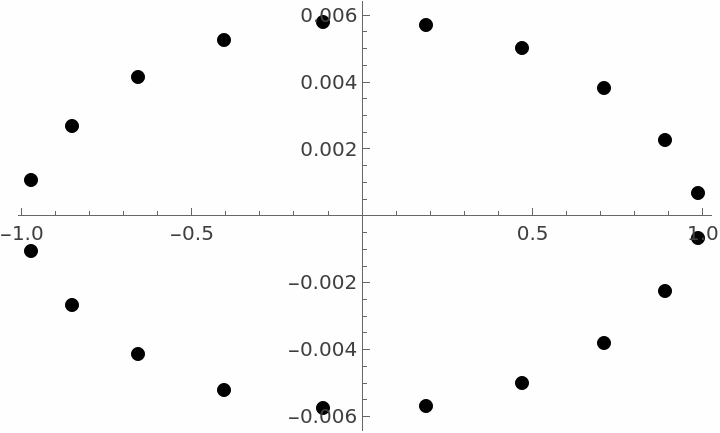}
    \caption{Zeros of the sum of Legendre polynomials}
    \label{fig:sumLegendre}
\end{figure}

\begin{theorem}\label{thm:mainthm}
Let $\left\{ S_n^{(\alpha)}(z) \right\}_{n=0}^\infty$  be the sequence of the sums of Gegenbauer polynomials:
\[
S^{(\alpha)}_n(z)=\sum_{k=0}^n C_k^{(\alpha)}(z).
\]
Then for $\alpha\geq1$, the zeros of $S^{(\alpha)}_n(z)$ lie on $(-1,1)$ for all large $n$.
\end{theorem} 

Since the case $\alpha = 1$ is a well-known result, we assume $\alpha>1$. Moreover, at the beginning of the next section, we argue that it suffices to prove Theorem \ref{thm:mainthm} for $1<\alpha<2$. In fact, we  show not only that the zeros of $S_n^{(\alpha)}(z)$ lie on $(-1,1)$, but also that the limiting density function for these zeros is the same as that of the Gegenbauer polynomial, which is \cite{szego}
\[
\frac{1}{\pi \sqrt{1-z^2}}, \quad z\in(-1,1).
\]
The paper is organized as follows. We find a complex integral representation for $S_{n}^{(\alpha)}(z)$ in Section 2 and an asymptotic formula for this sum in Section 3. We conclude the paper with a proof of Theorem \ref{thm:mainthm} in Section 4.  
\section{Generating function and integral representation}
We note that 
\begin{align*}
    \sum_{n=0}^\infty S^{(\alpha)}_n(z)t^n &= \sum_{n=0}^\infty \sum_{k=0}^nC_k^{(\alpha)}(z)t^n  \\
    &= \sum_{k=0}^\infty \sum_{n=k}^\infty C_k^{(\alpha)}(z)t^{n}  \\
    &= \sum_{k=0}^\infty t^k \sum_{n=0}^\infty C_k^{(\alpha)}(z)t^n  \\
    &= \frac{1}{(1-t)(1-2zt+t^2)^\alpha}. 
\end{align*}
We differentiate the first and last expressions with respect to $z$ to conclude 
\begin{align*}
\sum_{n=0}^\infty \frac{d(S_n^{(\alpha)}(z))}{dz}t^n &= \frac{2t\alpha}{(1-t)(1-2zt+t^2)^{\alpha+1}}    \\
&= 2\alpha\sum_{n=0}^\infty S_n^{(\alpha+1)}(z)t^{n+1},
\end{align*}
from which we deduce that for $n\ge 1$
\[
2\alpha S_{n-1}^{(\alpha+1)}(z) = \frac{d(S_n^{(\alpha)}(z))}{dz}.
\]
Note that if the zeros of a polynomial lie on $(-1,1)$, then so do the zeros of its derivative. Thus, it suffices to prove Theorem \ref{thm:mainthm} when $\alpha\in(1,2)$, an assumption we make for the remainder of this paper. 
In our approach, we will show that for each large $n\in\mathbb{N}$, the polynomial $S^{(\alpha)}_n(z)$ has $n$ zeros on the interval $(-1,1)$. For this purpose, we assume $z\in(-1,1)$ and let $z=\cos(\theta)$ for $\theta\in(0,\pi)$.   With the principal cut, we note that for small $t$
\[
-\pi<\Arg(1-te^{i\theta})+\Arg(1-te^{-i\theta})<\pi.
\]
Consequently
\[
(1-2zt+t^2)^\alpha=(1-2\cos(\theta)t+t^2)^\alpha=(1-te^{i\theta})^\alpha(1-te^{-i\theta})^\alpha,
\]
from which we deduce that 
\begin{equation} \label{eq:genfunc}
\sum_{n=0}^\infty S^{(\alpha)}_n(\cos\theta)t^n = \frac{1}{(1-t)(1-te^{i\theta})^\alpha(1-te^{-i\theta})^\alpha}.
\end{equation}
By the Cauchy Differentiation formula, we have 
\begin{equation} \label{eq:intsmallloop}
S_n^{(\alpha)}(\cos\theta)=\frac{1}{2\pi i} \oint_{|t|=\epsilon}\frac{dt}{t^{n+1}(1-t)(1-te^{i\theta})^\alpha(1-te^{-i\theta})^\alpha}.
\end{equation}
We note that, with the principal cut, the integrand
\begin{equation} \label{eq:gndef}
g_n(t):=\frac{1}{t^{n+1}(1-t)(1-te^{i\theta})^\alpha(1-te^{-i\theta})^\alpha}
\end{equation}
has singularities at $t=0$, $t=1$, and the two rays $r e^{\pm i \theta}$, $1\le r<\infty$ (see Figure \ref{fig:contourint}). 
\begin{figure}
    \centering
    \includegraphics[width=0.5\linewidth]{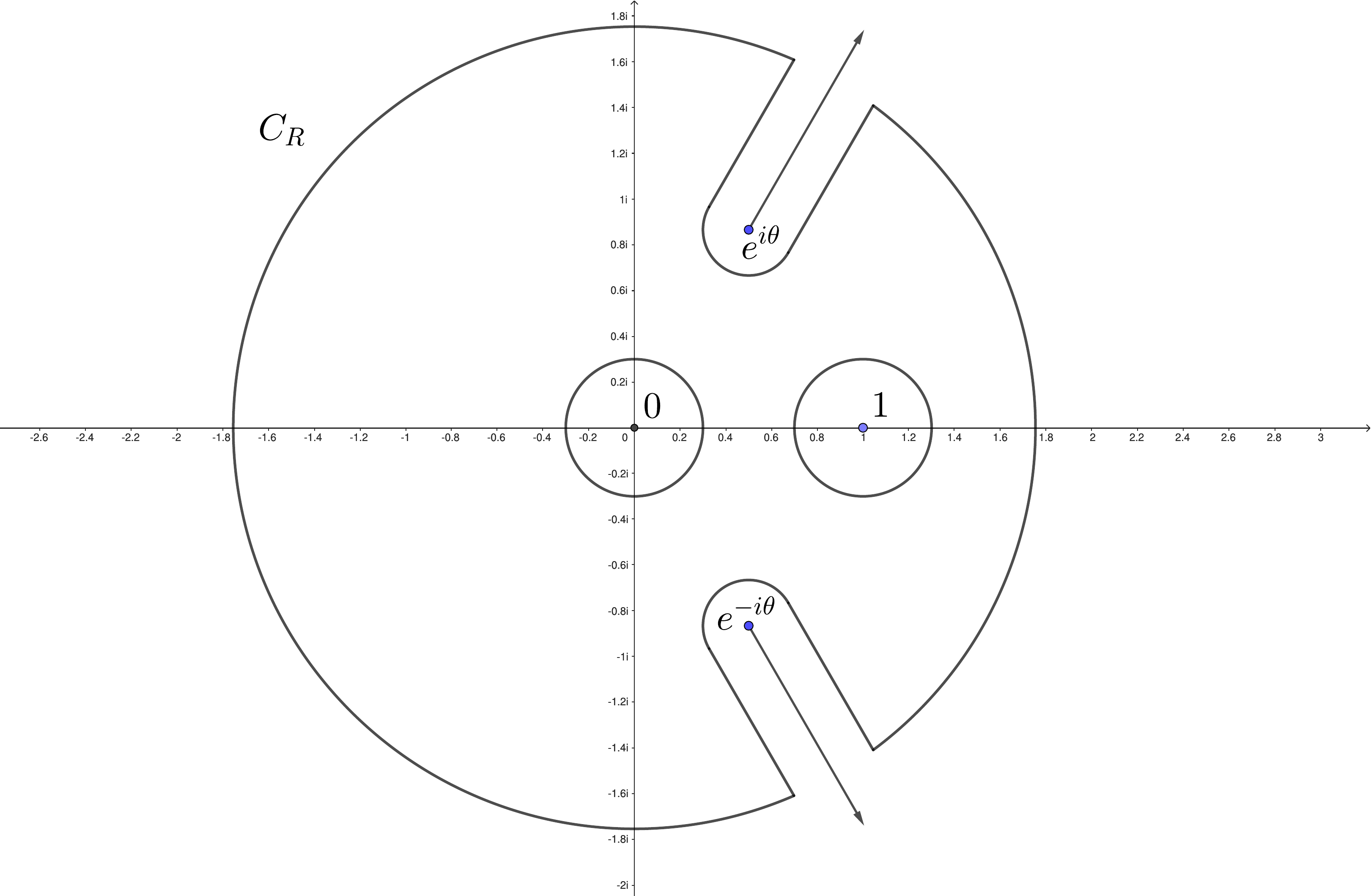}
    \caption{Contour of integration}
    \label{fig:contourint}
\end{figure}

On any circular arc $C_R:t=Re^{i\phi},a\le\phi\le b$, we have 
\begin{align*}
\left| \int_{C_R}   \frac{dt}{t^{n+1}(1-t)(1-te^{i\theta})^\alpha(1-te^{-i\theta})^\alpha} \right| \le \int_{a}^b\frac{d\phi}{R^n (R-1)^{2\alpha+1}}\rightarrow0
\end{align*}
as $R\rightarrow \infty$.  Thus  (\ref{eq:intsmallloop}) becomes
\begin{equation} \label{eq:cauchydeformed}
S_n^{(\alpha)}(\cos\theta)=\frac{1}{2\pi i} \left( \int_\gamma  g_n(t)dt - \int_{\overline{\gamma}}  g_n(t)dt\right) - \operatorname{Res}(g(t),1)
\end{equation}
where $\gamma$ is the contour around the ray $r e^{i\theta}$, $1\le r<\infty$,  and $\overline{\gamma}$ is its conjugate contour around $e^{-i\theta}$ (see Figure \ref{fig:contourint}).  From
\[
\overline{\int_\gamma g_n(t)dt}=\int_{\gamma}\overline{g_n(t)}\overline{dt}=\int_{\gamma}g_n(\overline{t})\overline{dt}=\int_{\overline{\gamma}}g_n(t)dt,
\]
the first term of the right side of  (\ref{eq:cauchydeformed}) is 
\[
\frac{1}{\pi}\Im\left(\int_\gamma g_n(t)dt\right).
\]
On the other hand, the second term of this expression is
\[
\operatorname{Res}(g(t),1)=\frac{1}{(1-e^{i\theta})^\alpha (1-e^{-i\theta})^\alpha} = \frac{1}{(2-2\cos\theta)^\alpha},
\]
from which we deduce that 
\[
S_n^{(\alpha)}(\cos\theta)=\frac{1}{\pi}\Im \left( \int_\gamma \frac{dt}{t^{n+1}(1-t)(1-te^{i\theta})^\alpha(1-te^{-i\theta})^\alpha}\right)-\frac{1}{(2-2\cos\theta)^\alpha}.    
\]
We make the substitution $t= e^{u+i\theta}$ to rewrite this equation as 
\begin{equation}\label{eq:s_nform}
\pi S_n^{(\alpha)}(\cos(\theta))= \Im (f_n(\theta))-\frac{\pi}{(2-2\cos \theta)^\alpha}
\end{equation}
where 
\begin{equation} \label{eq:fndef}
f_n(\theta)=\frac{1}{e^{in\theta}} \int_{+\infty}^{(0^-)} \frac{du}{(1-e^{u+i\theta})(1-e^u)^\alpha(1-e^{u+2i\theta})^\alpha e^{nu}}
\end{equation}
and the integral is taken over a clockwise Hankel contour around $[0,\infty)$ ($[0,\infty)$ are the only singularities of the integrand inside this contour).

\section{Asymptotic behavior}
We will show that $S_n^{(\alpha)}(\cos \theta)$ in (\ref{eq:s_nform}) has $n$ zeros on the interval $(0,\pi)$ using the Intermediate Value Theorem. To apply this theorem, we find a (finite) sequence of angles $\theta_1<\theta_2<\cdots<\theta_{n+1}$ on $(0,\pi)$ such that
\[
S_n^{(\alpha)}(\cos\theta_k)S_n^{(\alpha)}(\cos\theta_{k+1})<0.
\]
The values of $\theta_k$ will come from those where $f_n(\theta_k)\in i \mathbb{R}$. At those values of $\theta_k$ the sign of  $S_n^{(\alpha)}(\cos\theta_k)$ is the same as that of $\Im( f_n(\theta_k))$ if we can show
\[
|f_n(\theta_k)|>\frac{\pi}{(2-2\cos\theta_k)^{\alpha}}.
\]

Motivated by the approach above, in this section, we show that for large $n$, the inequality
\begin{equation}\label{eq:mainineq}
 \left| \int_{+\infty}^{(0^-)} \frac{du}{(1-e^{u+i\theta})(1-e^u)^\alpha(1-e^{u+2i\theta})^\alpha e^{nu}}\right| > \frac{\pi}{(2-2\cos\theta)^\alpha}
\end{equation}
holds for all $\theta \in (3\alpha/n,\pi-1/(n\sqrt{n}))$. We break this interval of $\theta$ into two intervals $[3\alpha/n,\delta)$ and $[\delta,\pi-1/(n\sqrt{n}))$ for a fixed (independent of $n$) small $\delta$.

\subsection{Case $\theta \in[\delta,\pi-1/(n\sqrt{n}))$} In this subsection, we will show that for any small $\delta$: 
\begin{equation} \label{eq:asymptoticlargetheta}
\int_{+\infty}^{(0^-)} \frac{du}{(1-e^{u+i\theta})(1-e^u)^\alpha(1-e^{u+2i\theta})^\alpha e^{nu}} \sim 
 \frac{2i\pi n^{\alpha-1}}{\Gamma(\alpha)(1-e^{i\theta})(1-e^{2i\theta})^\alpha }
\end{equation}
uniformly in $\theta \in [\delta,\pi-1/(n\sqrt{n}))$. This uniform convergence will imply (\ref{eq:mainineq}) for large $n$ and $\theta \in [\delta,\pi-1/(n\sqrt{n}))$. For each $n$, we can choose a  Hankel contour  $\gamma$ as the union of two rays
\[
l_+(x)=\frac{1}{\sqrt{n}}+i\frac{1}{n\sqrt{n}}+x,\quad 0\le x<\infty
\]
and 
\[
l_{-}(x)=\frac{1}{\sqrt{n}}-i\frac{1}{n\sqrt{n}}-x,\quad -\infty<x\le 0,
\]
and the curve $\lambda$ which starts from $1/\sqrt{n} -i/(n\sqrt{n})$, winds around the origin with a clockwise orientation, and ends at $1/\sqrt{n}+i/(n\sqrt{n})$.  Let $R(\gamma)$ be the region containing $[0,\infty)$ whose boundary is $\gamma$ (see Figure \ref{fig:Hankel}).  

\begin{figure}
    \centering
    \includegraphics[width=0.5\linewidth]{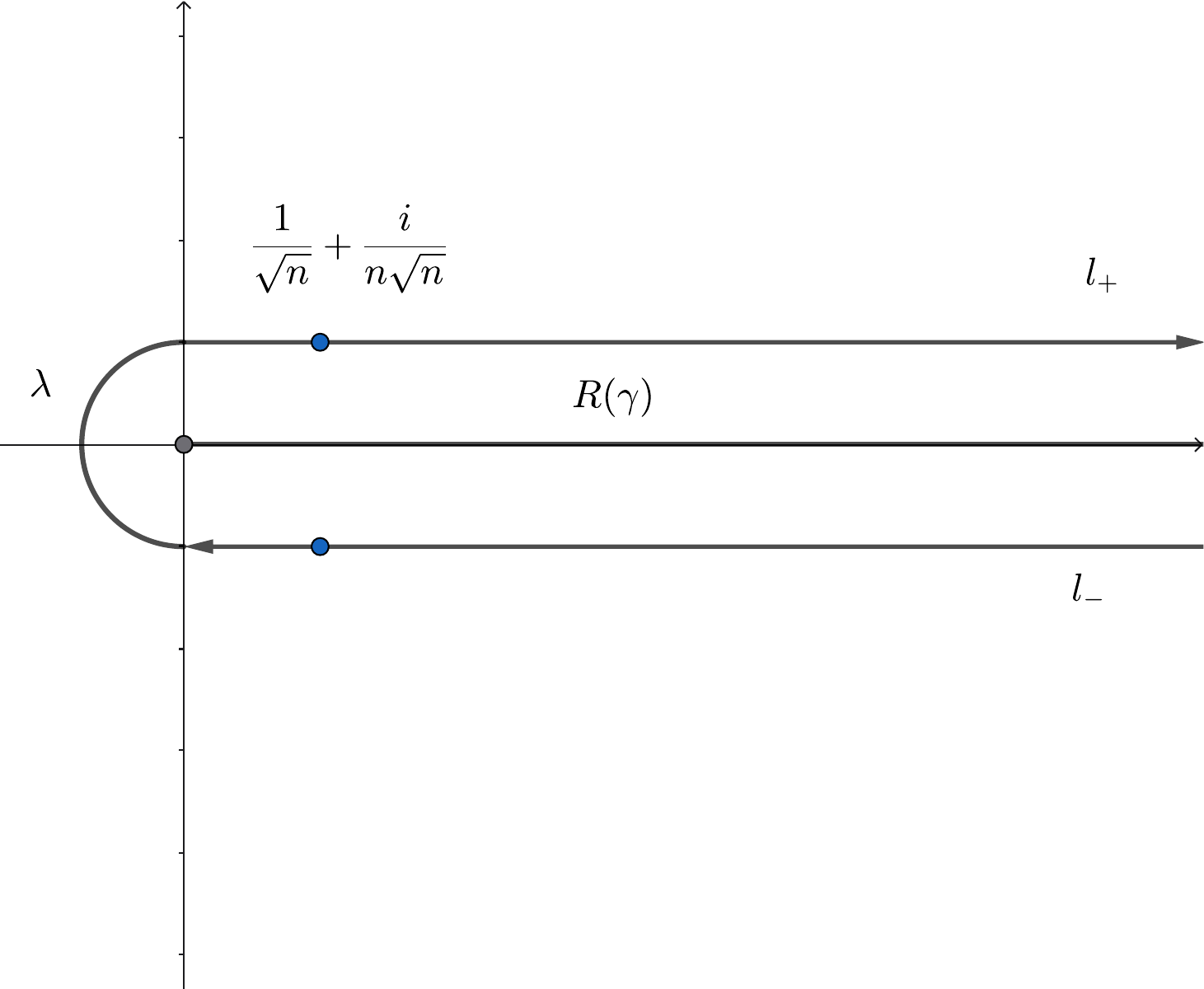}
    \caption{Hankel contour}
    \label{fig:Hankel}
\end{figure}

To show that $\gamma$ is valid, we prove that besides $[0,\infty)$, all the singularities of the integrand lie outside the closure $\overline{R(\gamma)}$  in the lemma below.

\begin{lemma}
If $1-e^{z+i\theta}=0$ or $1-e^{z+2i\theta} \in [-\infty,0)$, then $z\notin \overline{R(\gamma)}$ for large $n$ and $\theta\in [\delta,\pi-1/(n\sqrt{n}))$. 
\end{lemma}
\begin{proof}
 Suppose $1-e^{z+2i\theta} \in [-\infty,0)$. This is equivalent to $\Im(z)+2\theta =2k\pi$, $k\in \mathbb{Z}$, and $1-e^{\Re z}<0$.  If $z\in \overline{R(\gamma)}$, then $|\Im z|\le 1/(n\sqrt{n})$. This contradicts to $\Im(z)+2\theta =2k\pi$  for $k\in \mathbb{Z}$ and $\theta\in(\delta,1/(n\sqrt{n}))$. We apply the same argument to the case $1-e^{z+i\theta}$  and complete the proof of this lemma. 
\end{proof}

The integrals over the two tails $l_{+}(x)$  and $l_{-}(x)$ is at most
\[
\int_{0}^\infty \frac{dx}{(e^{1/\sqrt{n}}-1)^{2\alpha+1}e^{\sqrt{n}+nx}}=\mathcal{O} \left( \frac{n^{\alpha-1/2}}{e^{\sqrt{n}}} \right).
\]
Next, we consider the integral over the curve $\lambda$, which is denoted by
\[
\int_{1/\sqrt{n}\pm i/(n\sqrt{n})}^{(0^-)} \frac{du}{(1-e^{u+i\theta})(1-e^u)^\alpha(1-e^{u+2i\theta})^\alpha e^{nu}}.
\]
For $u$ on this curve and $\theta\in [\delta,\pi-1/(n\sqrt{n}))$, we have
\[
  1-e^{u+i\theta}=1-e^{i\theta}+\mathcal{O}(u)=(1-e^{i\theta})(1+\mathcal{O}(u)) 
\]
where the big-Oh constant here (and in the whole paper) is independent of $\theta$.  Together with similar identities $1-e^u=(-u)(1+\mathcal{O}(u))$ and $$1-e^{u+2i\theta}=(u+2i\theta)(1+\mathcal{O}(u)),$$  the integral over the curve $\lambda$  becomes
\begin{equation}\label{eq:intsplitmainterm}
\frac{1}{(1-e^{i\theta})(1-e^{2i\theta})^\alpha } \left( \int_{1/\sqrt{n}\pm i/(n\sqrt{n})}^{(0^-)} \frac{du}{(-u)^\alpha e^{nu}}+\mathcal{O}\left(\int_{1/\sqrt{n}\pm i/(n\sqrt{n})}^{(0^-)} \frac{|du|}{|u|^{\alpha-1} |e^{nu}|}\right) \right).
\end{equation}
The lemma below provides a bound for the big-Oh term of this expression. 
\begin{lemma} \label{lem:boundtailgamma}
    For $1<\alpha<2$, we have 
    \[
    \int_{1/\sqrt{n}\pm i/(n\sqrt{n})}^{(0^-)} \frac{|du|}{|u|^{\alpha-1} |e^{nu}|}=\mathcal{O}(n^{\alpha-2}).
    \]
    as $n\rightarrow\infty$.
\end{lemma}
\begin{proof}
   We note that the curve $\lambda$ consists of a semicircle of radius $1/n\sqrt{n}$ and two line segments $t\pm i/(n\sqrt{n})$, $0\le t\le 1/\sqrt{n}$ (see Figure \ref{fig:Hankel}).
   On the semicircle of radius $1/n\sqrt{n}$, we have 
   \[
   \frac{1}{|e^{nu}|}=\frac{1}{e^{n\Re u}} \le e^{1/\sqrt{n}}
   \]
   from which we deduce that the integral over this semicircle is at most
   \[
   (n\sqrt{n})^{\alpha-2} e^{1/\sqrt{n}}\pi.
   \]
   On the two segments $u=t\pm i/(n\sqrt{n})$, $0\le t \le 1/\sqrt{n}$,  
   \[
   \frac{1}{|u|^{\alpha-1}|e^{nu}|} \le \frac{1}{t^{\alpha-1}e^{nt}}
   \]
   and thus the integrals on these two line segments are at most
   \[
   \int_{0}^{1/\sqrt{n}}\frac{dt}{t^{\alpha-1}e^{nt}}\underbrace{=}_{t\rightarrow t/n}n^{\alpha-2}\int_0^{\sqrt{n}}\frac{dt}{t^{\alpha-1}e^t}\le n^{\alpha-2} \Gamma(2-\alpha).
   \]
   The lemma follows. 
   
\end{proof}
We now turn our attention to the main term: 
\begin{equation}\label{eq:maintermchangevariable}
\int_{1/\sqrt{n}\pm i/(n\sqrt{n})}^{(0^-)} \frac{du}{(-u)^\alpha e^{nu}}\underbrace{=}_{u\rightarrow u/n} n^{\alpha-1} \int_{\sqrt{n}\pm i/\sqrt{n}}^{(0^-)} \frac{du}{(-u)^\alpha e^{u}} .
\end{equation}
To convert the integral of the last expression to the Gamma function, which has the complex integral representation \cite{ww}
\[
\Gamma(z) = \frac{1}{2i\sin(\pi z)}\int_{+\infty}^{(0^-)}e^{-u}(-u)^{z-1}du,\quad z\notin \mathbb{Z}^-,
\]
we  bound the tail integrals
\[
\int_{\sqrt{n}\pm i/\sqrt{n}}^{\infty\pm i/\sqrt{n}} (-u)^{-\alpha}e^{-u}du=\mathcal{O}\left( \int_{\sqrt{n}}^{\infty} u^{-\alpha} e^{-u} \right) =\mathcal{O}\left( n^{-\alpha/2} e^{-\sqrt{n}}\right),
\]
where the last equation comes from the upper incomplete gamma function 
\[
\Gamma(z,x)=\int_{x}^\infty t^{z-1}e^{-t}dt.
\]
which is asymptotic to \cite{temme}
\[
\Gamma(z,x)\sim x^{z-1}e^{-x},
\]
for large $x$. 
With this bound the Euler's reflection formula
\[
\sin(\pi \alpha)\Gamma(1-\alpha)  = \frac{\pi}{\Gamma(\alpha)},
\]
we conclude that
\begin{equation} \label{eq:intaroundorigin}
\int_{\sqrt{n}\pm i/\sqrt{n}}^{(0^-)} \frac{du}{(-u)^\alpha e^u} \sim \frac{2i\pi}{\Gamma(\alpha)}.
\end{equation}
Thus, the asymptotic (\ref{eq:asymptoticlargetheta}) follows from (\ref{eq:intsplitmainterm}) and (\ref{eq:maintermchangevariable}).

\subsection{Case $\theta \in (3\alpha/n,\delta)$} In this subsection, we will prove that for fixed (independent of $n$) and small $\delta$, the inequality (\ref{eq:mainineq}) holds for all large $n$ and $\theta \in (3\alpha/n,\delta)$. As in the previous case, we split the contour of integration on the left side of (\ref{eq:mainineq}) into three curves. The first curve starts from $1/\sqrt{n}-i/n\sqrt{n}$, winds around the origin with a clockwise orientation, and ends at $1/\sqrt{n}+i/n\sqrt{n}$.  The other two curves are the two rays
\begin{align*}
  &1/\sqrt{n}+i/n\sqrt{n}+x,0\le x < \infty, \\
  &1/\sqrt{n}-i/n\sqrt{n}-x,-\infty < x \le 0.
\end{align*}
On the first curve, $u$ is small, so we have 
\[
1-e^{u+i\theta}=(-u-i\theta)(1+\mathcal{O}(u+\theta))
\] and \[1-e^{u+2i\theta}=(-u-2i\theta)(1+\mathcal{O}(u+\theta)).\]
Thus, the integrand of the integral on the right side of (\ref{eq:fndef}) is
\[
\frac{1+\mathcal{O}(u+\theta)}{(-u-i\theta)(-u-2i\theta)^\alpha (-u)^\alpha e^{nu}}=\frac{1+\mathcal{O}(u+\theta)}{2^{\alpha}(-i\theta)^{\alpha+1}(u/i\theta+1)(u/(2i\theta)+1)^\alpha(-u)^\alpha e^{nu}}.
\]

We let $u=v/n$ and $\theta = c /n$, for $3\alpha \le c < \delta n$, and deduce that
\begin{align}
&\int_{1/\sqrt{n}\pm i/n\sqrt{n}}^{(0^-)} \frac{1+\mathcal{O}(u+\theta)}{2^{\alpha}(-i\theta)^{\alpha+1}(u/i\theta+1)(u/(2i\theta)+1)^{\alpha}(-u)^{\alpha}e^{nu}}du \nonumber \\
= &\frac{n^{2\alpha}(1+\mathcal{O}(1/\sqrt{n}+\theta))}{2^{\alpha}(-ic)^{\alpha+1}}\int_{\sqrt{n}\pm i/\sqrt{n}}^{(0^-)} \frac{dv}{(v/ic+1)(v/(2ic)+1)^{\alpha}(-v)^{\alpha}e^{v}} . \label{eq:mainintegral}
\end{align}
We rewrite the integral in this expression as
\begin{align}
J(\theta) &:= \int_{\sqrt{n}\pm i/\sqrt{n}}^{(0^-)} \frac{dv}{(v/ic+1)(v/(2ic)+1)^{\alpha}(-v)^{\alpha}e^{v}} \nonumber \\ 
&= \int_{\sqrt{n}\pm i/\sqrt{n}}^{(0^-)} \frac{dv}{(-v)^{\alpha}e^{v}}+\int_{\sqrt{n}\pm i/\sqrt{n}}^{(0^-)} \left( \frac{1}{(v/ic+1)(v/(2ic)+1)^\alpha } -1\right)\frac{dv}{(-v)^\alpha e^v}. \label{eq:mainterm}
\end{align}
To obtain an upper bound for the second integral of the above expression, we break up this integral into 
\begin{align}
& \int_{\epsilon\pm i/ \sqrt{n}}^{(0^-)}  \left( \frac{1}{(v/ic+1)(v/(2ic)+1)^\alpha } -1\right)\frac{dv}{(-v)^\alpha e^v} \nonumber \\
+ & \int_{\epsilon \pm i/\sqrt{n}}^{\sqrt{n}\pm i/\sqrt{n}} \left( \frac{1}{(v/ic+1)(v/(2ic)+1)^\alpha } -1\right)\frac{dv}{(-v)^\alpha e^v} \label{eq:smallerterm}
\end{align}
for a small $\epsilon$, independent of $n$ (see Figure \ref{fig:contour}). 

\begin{figure} \label{fig:contour}
    \centering
    \includegraphics[width=0.5\linewidth]{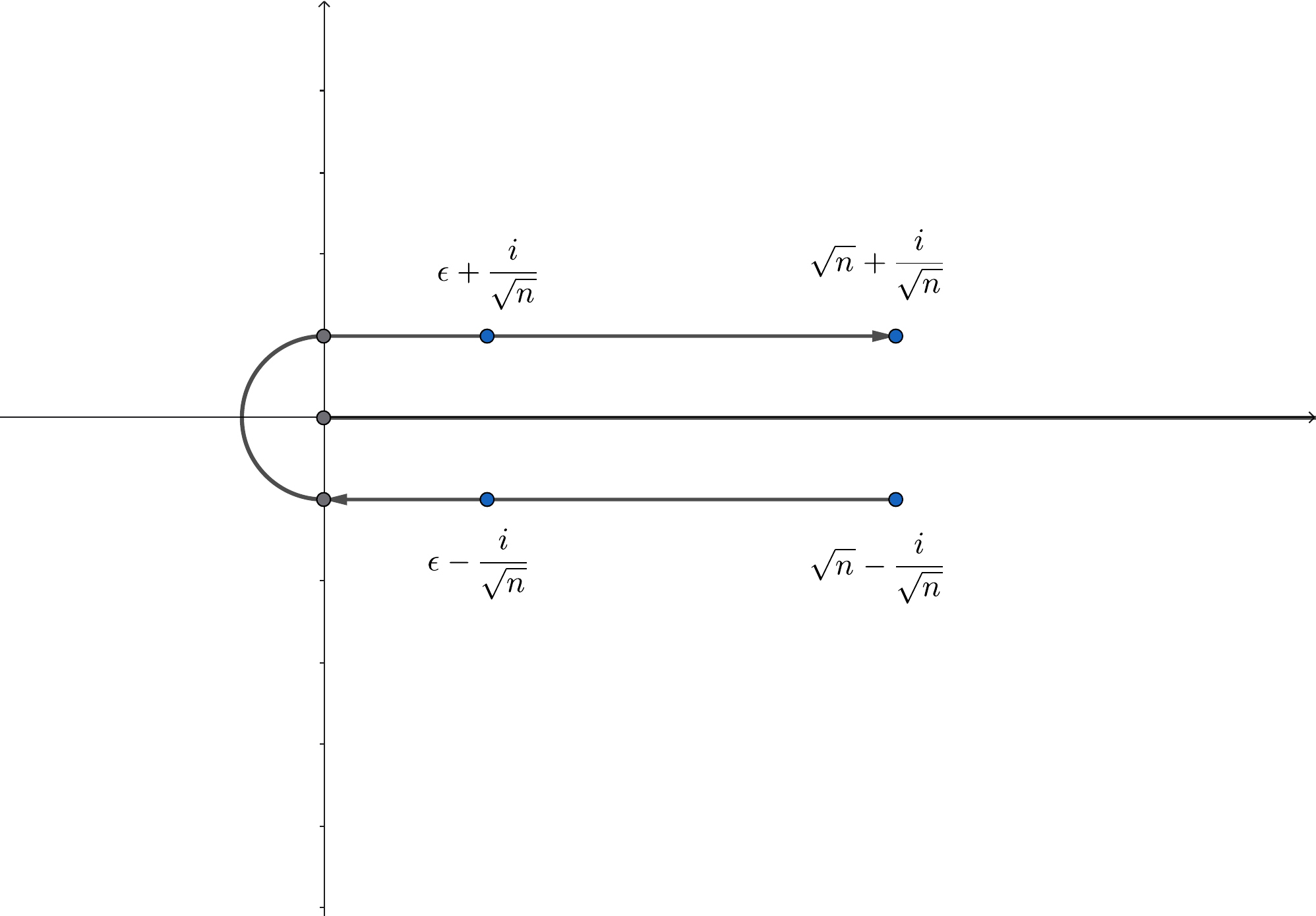}
    \caption{Contour of integration}
    \label{fig:placeholder}
\end{figure}

\begin{lemma}\label{lem:boundintegrand}
    If $z\in i\mathbb{R}$, then 
    \[
    \left| \frac{1}{(1+z)(1+z/2)^\alpha}-1 \right| \le  \left(1+\frac{\alpha}{2} \right)|z|.
    \]
   On the other hand, if $z\in\mathbb{C}$ and $|z|<1$, then the left side is at most
  \[
\frac{|z|}{(1-|z|)^2(1-|z|/2)^\alpha}+\frac{\alpha|z|}{2(1-|z|)(1-|z|/2)^{\alpha+1}}.
\]
\end{lemma}
\begin{proof}
If we let 
\[
f(z)=\frac{1}{(1+z)(1-z/2)^\alpha},
\]
then 
\[
f(z)-1=\int_l f'(u)du
\]
where the integral is taken over the  line segment $l(t)=tz $, $0\le t\le 1$. The modulus of the right side is at most
\[
\int_l |f'(u)| |du|=\int_l \left| \frac{1}{(1+u)^2(1-u/2)^\alpha } +\frac{\alpha}{2(1+u)(1-u/2)^{\alpha+1}}\right| |du|.
\]
If $z\in i\mathbb{R}$, then  $u$ is purely imaginary and  the integrand on the right side is at most
\[
1+\frac{\alpha}{2}
\]
from which the case $z\in i\mathbb{R}$ follows. On the other hand, if $|z|<1$, then this integrand is at most

\[
\frac{1}{(1-|z|)^2(1-|z|/2)^\alpha}+\frac{\alpha}{2(1-|z|)(1-|z|/2)^{\alpha+1}}
\]
and we complete this case. 

\end{proof}
If $v$ lies in the contour of the first integral in (\ref{eq:smallerterm}), then $|v/ic|<2 \epsilon<1$  for large $n$ and small $\epsilon$. By Lemma \ref{lem:boundintegrand}, we have that
\[
\left| \frac{1}{(v/ic+1)(v/(2ic)+1)^\alpha } -1\right|
\]
is at most
\[
\left( \frac{1}{(1-2\epsilon)^2(1-\epsilon)^{\alpha}}+\frac{\alpha}{2(1-2\epsilon)(1-\epsilon)^{\alpha+1}} \right)\frac{|v|}{c}.
\]
We deduce that the modulus of the first integral in   (\ref{eq:smallerterm}) is $\mathcal{O}(\epsilon^{2-\alpha})$, which is small for small $\epsilon$. 

We now consider  the second  integral in   (\ref{eq:smallerterm}), which is the integral over the segment
\[
i/\sqrt{n}+x, \epsilon \le x \le \sqrt{n}
\]
minus the integral over the conjugate of this segment. Since $\epsilon$ is independent of $n$, we have that as $n\rightarrow\infty$,  the integrand of this integral in this segment and its conjugate are 
\[
\left( \frac{1}{(x/ic+1)(x/(2ic)+1)^\alpha } -1\right)\frac{1+\mathcal{O}(1/\sqrt{n})}{x^\alpha e^{\pm i\alpha \pi} e^x}.
\]
Thus the second integral of (\ref{eq:smallerterm}) is
\[
-2i\sin(\pi \alpha)\int_{\epsilon}^{\sqrt{n}} \left( \frac{1}{(x/ic+1)(x/(2ic)+1)^\alpha } -1\right)\frac{1+\mathcal{O}(1/\sqrt{n})}{x^\alpha  e^x}dx.
\]
We apply Lemma \ref{lem:boundintegrand}, 
\[
\left| \frac{1}{(x/ic+1)(x/(2ic)+1)^\alpha } -1\right|\le \left(1+\frac{\alpha}{2}\right)\frac{x}{c},
\]
to conclude that the modulus of the main term of this expression is at most
\[
\left(1+\frac{\alpha}{2}\right)\frac{2|\sin(\pi \alpha)|}{c}\int_0^\infty x^{1-\alpha}e^{-x}dx =\left(1+\frac{\alpha}{2}\right)\frac{2|\sin(\pi \alpha)|\Gamma(2-\alpha)}{c}.
\]
As $1<\alpha<2$, we have $|\sin(\pi \alpha)| = \sin(\pi(\alpha-1))$ from which the last expression becomes
\[
\left( 1+ \frac{\alpha}{2}\right) \frac{2\pi}{c\Gamma(\alpha-1)}=\left( 1+ \frac{\alpha}{2}\right) \frac{2\pi (\alpha-1)}{c\Gamma(\alpha)}.
\]
By (\ref{eq:intaroundorigin}),  the first integral of (\ref{eq:mainterm}) is asymptotic to
\[
\frac{2i\pi}{\Gamma(\alpha)}.
\]
Thus the imaginary part of  (\ref{eq:mainterm})  is at least 
\begin{equation} \label{eq:Imlowerbound}
\frac{2\pi}{\Gamma(\alpha)}-\left( 1+ \frac{\alpha}{2}\right) \frac{2\pi (\alpha-1)}{c\Gamma(\alpha)} + \mathcal{O}(\epsilon^{2-\alpha})> \frac{2\pi}{\Gamma(\alpha)} \left( 1-\frac{2(\alpha-1)}{c} \right)    
\end{equation}
and consequently, the modulus of the main term in the last expression of (\ref{eq:mainintegral}) is more than
\[
\frac{2\pi n^{2\alpha}}{2^\alpha c^{\alpha+1}\Gamma(\alpha)}\left( 1-\frac{2(\alpha-1)}{c} \right).
\]
We claim that for $c\ge 3 \alpha$, this expression is at least 
\[
\frac{\pi n^{2\alpha}}{c^{2\alpha}}.
\]
Indeed, this claim is equivalent to 
\[
\left( \frac{c}{2} \right)^{\alpha-1} - (\alpha-1)\left( \frac{c}{2} \right)^{\alpha-2}\ge \Gamma(\alpha).
\]
This inequality holds as the derivative of the function 
\[
x^{\alpha-1}-(\alpha-1)x^{\alpha-2}-\Gamma(\alpha)
\]
is 
\[
(\alpha-1)x^{\alpha-2}-(\alpha-1)(\alpha-2)x^{\alpha-3} > 0
\]
for $x \in [ 3 \alpha/2,\infty)$ and the value of this function at $x=3\alpha/2$ is
\[
(3\alpha/2)^{\alpha-2} (\alpha/2+1) -\Gamma(\alpha),
\]
which is positive by the lemma below.
\begin{lemma}
For $1<\alpha<2$, we have 
\[
(3\alpha/2)^{\alpha-2}(\alpha/2+1) >  \Gamma(\alpha).
\]
\end{lemma}
\begin{proof}
The claimed inequality is equivalent to 
 \[
 (\alpha-2)\ln(3\alpha/2)+\ln(\alpha/2+1)>\ln\Gamma(\alpha). 
 \]
 Since the two sides are equal when $\alpha=1$, it suffices to show the derivative (in $\alpha$) of the left side is larger than that of the right side. The difference between these two derivatives is 
 \[
 \ln(3\alpha/2)+ \frac{\alpha-2}{\alpha}+\frac{1}{2+\alpha}+\gamma-\frac{\Gamma'(\alpha)}{\Gamma(\alpha)}.
 \]
We apply the inequality \cite{alzer}
\[
\frac{\Gamma'(\alpha)}{\Gamma(\alpha)}<\ln(\alpha)-\frac{1}{2\alpha}
\]
 to deduce that the expression above is at least
 \[
\ln(3/2)+\gamma +1 +\frac{1}{2+\alpha} -\frac{3}{2\alpha}>0
 \]
 and the lemma follows.
 
\end{proof}
As 
\[
\frac{\pi}{(2-2\cos\theta)^\alpha} \sim \frac{\pi}{\theta^{2\alpha}}=\frac{\pi n^{2\alpha}}{c^{2\alpha}}
\]
for small $\theta$, we conclude (\ref{eq:mainineq}). 

\section{Zero distribution of $S^{(\alpha)}_n(z)$}
We will show that, for large $n$, all zeros of $S_n^{(\alpha)}(z)$ lie on $(-1,1)$ by showing that this polynomial has at least $n$ (the degree of this polynomial) zeros on this interval. We recall from (\ref{eq:s_nform}) that $\pi S_n^{(\alpha)}(\cos(\theta))$ is the imaginary part of 
\[
g_n(\theta)=f_n(\theta)-\frac{\pi i}{(2-2\cos \theta)^\alpha}.
\]
By (\ref{eq:mainineq}), if $\theta^{\pm}$ satisfy $f_n(\theta^{\pm}) \in i\mathbb{R^{\pm}}$, then $g_n(\theta^{\pm})\in i\mathbb{R}^{\pm}$. By the Intermediate Value Theorem, there is a zero of $g_n(\theta)$ in the interval between $\theta^+$ and $\theta^-$. To count the number of such intervals, we will compute the change in argument of the curve 
\[
f_n(\theta)=e^{-in\theta}\int_{+\infty}^{(0^-)}\frac{du}{(1-e^{u+i\theta})(1-e^u)^\alpha(1-e^{u+2i\theta})^\alpha e^{nu}}, 
\]  
for $\theta \in (3\alpha/n,\pi-1/(n\sqrt{n}))$.  

The change of argument of the first factor $e^{-in\theta}$ of $f_n(\theta)$  on $(3\alpha/n,\pi-1/(n\sqrt{n}))$ is 
\begin{equation} \label{eq:deltaArgexp}
\Delta \arg_{(3\alpha/n,\pi-1/(n\sqrt{n}))} e^{-in\theta}=-n\pi + 3\alpha+\frac{1}{\sqrt{n}}.
\end{equation}
Let 
\[
r_n(\theta)= \int_{+\infty}^{(0^-)}\frac{du}{(1-e^{u+i\theta})(1-e^u)^\alpha(1-e^{u+2i\theta})^\alpha e^{nu}},
\]
be the second factor of $f_n(\theta)$. We recall that for $\theta \in I:=(\delta,\pi-1/(n\sqrt{n}))$, $r_n(\theta)$ is asymptotic to the right side of (\ref{eq:asymptoticlargetheta}). Thus the change in argument of $r_n(\theta)$ on $I$ is 
\[
-\Delta \arg_I (1-e^{i\theta}) -\alpha \Delta \arg_I (1-e^{2i\theta})+o(1). 
\]
Since $1-e^{i\theta}$ and $1-e^{2i\theta}$ lie in the open right-half plane for $\theta \in I$, their changes in argument are the differences of principal angles at the two endpoints which are $-\Arg(1-e^{i\delta})+o(1)$ and $\pi/2-\Arg(1-e^{2i\delta})+o(1)$ respectively. Thus
\begin{align*}
\Delta \arg_I r_n(\theta)&= \Arg(1-e^{i\delta})-\alpha \pi/2+\alpha\Arg(1-e^{2i\delta})+o(1) \\
&=-\pi/2 -\alpha\pi+\mathcal{O}(\delta).
\end{align*}
For $\theta\in (3\alpha/n,\delta)$, $r_n(\theta)$ is (\ref{eq:mainintegral}), where the imaginary part of the integral, $J(\theta)$, in (\ref{eq:mainintegral}) is at least (\ref{eq:Imlowerbound}), which is positive as $c\ge 3\alpha$. Thus, the change in argument of $r_n(\theta)$ on $(3\alpha/n,\delta)$ is 
\[
\Delta \arg_{(3\alpha/n,\delta)} J(\theta)=\Arg J(\delta)-\Arg J(3\alpha/n).
\]
As 
\[
r_n(\delta)\sim \frac{3 i\pi n^{\alpha-1}}{\Gamma(\alpha)(1-e^{i\delta})(1-e^{2i\delta})^\alpha}
\]
and 
\[
r_n(\delta) \sim \frac{n^{2\alpha} J(\delta)}{2^\alpha (-i\delta/n)^{\alpha+1}}
\]
by (\ref{eq:asymptoticlargetheta}) and (\ref{eq:mainintegral}) respectively, we have 
\[
\Arg J(\delta) = \frac{\pi}{2}+\mathcal{O}(\delta).  
\]
Consequently, as $J(\theta)$ lies in the upper half plane for $3\alpha<\theta<\delta$, 
\[
|\Delta \arg_{(3\alpha/n,\delta)} J(\theta)|<\pi/2+\mathcal{O}(\delta).
\]
We conclude from 
\[
\Delta \arg_{(3\alpha/n,\pi-1/(n\sqrt{n}))} r_n(\theta) = \Delta \arg_I r_n(\theta)+\Delta \arg_{(3\alpha/n,\delta)} r_n(\theta)
\]
and (\ref{eq:deltaArgexp}) that  
\begin{equation}\label{eq:changeargf_n}
\Delta \arg_{(2\alpha/n,\pi-1/(n\sqrt{n}))}f_n(\theta)= -n\pi+3\alpha -\pi/2-\alpha \pi +A
\end{equation}
where $A<\pi/2+\mathcal{O}(\delta)$.

Let  $\theta_1 <\theta_2 <\ldots <\theta_m$  be the (finite) sequence of angles in $(3\alpha/n,\pi-1/(n\sqrt{n}))$  such that $f_n(\theta_k)\in i\mathbb{R}$ ,  for $1\le k\le m$, and $f_k (\theta) f_{k+1}(\theta)<0$ (for $k<m$). By (\ref{eq:changeargf_n})
\[
m \ge \left \lfloor \frac{|\Delta \arg_{(3\alpha/n,\pi-1/(n\sqrt{n}))}f_n(\theta)|}{\pi} \right \rfloor\ge n
\]
for sufficiently small $\delta$. From (\ref{eq:s_nform}) and (\ref{eq:mainineq}), we conclude from the Intermediate Value Theorem that $S_n^{(\alpha)}(\cos\theta)$  has at least a zero on each interval $(\theta_k,\theta_{k+1})$ for $1\le k<m$. Thus $S_n ^{(\alpha)}(\cos \theta)$ has at least $n-1$ zeros on $(3\alpha/n,\pi-1/(n\sqrt{n}))$, each of which gives a zero of $S_n^{(\alpha)}(z)$ on $(-1,1)$. As the degree of 
\[
S_n^{(\alpha)}(z)=\sum_{k=0}^n C_k ^{(\alpha)}(z)
\]
is $n$, the one remaining zero of $S_n^{(\alpha)}(z)$ must be real. The claim that all zeros of $S_n^{(\alpha)}(z)$ lie on $(-1,1)$ follows from the lemma below.
\begin{lemma}
 If  $1 < \alpha<2$, then    $S_n^{(\alpha)}(x)\ne 0$  for $x\in(-\infty,1]\cup[1,\infty)$.
\end{lemma}
\begin{proof}
In the case $x>1$, we have  
\[
1-2xt+t^2 =(1-t_1 t)(1-t_2 t)
\]
where $t_1$ and $t_2$ are positive real numbers.  We deduce from the generating function
\[
\sum_{n=0}^\infty C^{(\alpha)}_n(x)t^n = (1-t_1 t)^{-\alpha}(1-t_2t)^{-\alpha}
\]
and the binomial expansions that $C_n^{(\alpha)}(x)>0$, $\forall n$.  Thus $S_n^{(\alpha)}(x)\ne 0$ for $x\ge 1$. Moreover, in the case $x\ge 1$, we conclude from the recurrence 
\[
(n+1)C^{(\alpha)}_{n+1}(x)=2(n+\alpha)xC_n^{(\alpha)}(x)-(n+2\alpha-1)C^{(\alpha)}_{n-1}(x), \quad n\ge 1,
\]
that
\[
(n+1)C^{(\alpha)}_{n+1}(x)-(n+1)C^{(\alpha)}_{n}(x)\ge (n+2\alpha-1)C^{(\alpha)}_n(x)-(n+2\alpha-1)C^{(\alpha)}_{n-1}(x).
\]
Thus, by induction, we have $C^{(\alpha)}_{n+1}(x)> C^{(\alpha)}_{n}(x)$ for $x\ge 1$.  

We recall that $C_n^{(\alpha)}(x)$ is an even/odd polynomial if $n$ is even/odd. Thus, in the case $x\le -1$, we have
\[
S^{(\alpha)}_n(x)=( C_n^{(\alpha)}(|x|)-C_{n-1}^{(\alpha)}(|x|))+\cdots + (C_2^{(\alpha)}(|x|)-C_1^{(\alpha)}(|x|))+C_0^{(\alpha)}(|x|)>0
\]
if $n$ is even and 
\[
S^{(\alpha)}_n(x)=-( C_n^{(\alpha)}(|x|)-C_{n-1}^{(\alpha)}(|x|))- \cdots - (C_1^{(\alpha)}(|x|)-C_0^{(\alpha)}(|x|))<0
\]
if $n$ is odd. The proof of this lemma is complete. 
\end{proof}

For the remainder of this paper, we will find the limiting probability density function of $S_n^{(\alpha)}(z)$ on $(-1,1)$. For each $n>0$ , $x\in(-1,1)$, and small $\epsilon>0$, let $N_{n,\epsilon}(z)$ be the number of zeros of $S_{n}^{(\alpha)}(x)$
on the interval $(x,x+\epsilon)$. The limiting probability density function of the zeros of $S_{n}^{(\alpha)}(x)$ evaluated at each
$x\in(-1,1)$ is defined as 
\[
\lim_{\epsilon\rightarrow0}\frac{1}{\epsilon}\lim_{n\rightarrow\infty}\frac{N_{n,\epsilon}(x)}{n}.
\]
For each $x=\cos\theta \in(-1,1)$, we have  
\[
\pi S_n^{(\alpha)}(\cos(\theta))=\Im \left( f_n(\theta)\right)-\frac{\pi}{(2-2\cos\theta)}
\]
where $f_n(\theta)=e^{-in\theta} r_n(\theta)$ and
\[
r_n(\theta)\sim  \frac{2i\pi n^{\alpha-1}}{\Gamma(\alpha)(1-e^{i\theta})(1-e^{2i\theta})^\alpha }.
\]
Thus
\begin{align*}
N_{n,\epsilon} &= \frac{|\Delta \arg_{(\arccos(x+\epsilon),\arccos(x))}f_n(\theta)|}{\pi}+\mathcal{O}(1) \\
&=\frac{n}{\pi}(\arccos(x)-\arccos(x+\epsilon))+\mathcal{O}(1).
\end{align*}
We conclude that the limiting probability density function of $S_n^{(\alpha)}(x)$ on $(-1,1)$ is 
\[
-\frac{1}{\pi}\frac{d}{dx}\arccos x=\frac{1}{\pi\sqrt{1-x^2}}.
\]

\subsection*{Data availability statement}

The authors confirm that the data supporting the findings of this
study are available within the article and its references below.

\subsection*{Funding and/or Competing interests}

The authors confirm that they have no conflicts of interest and no
funding was received to assist with the preparation of this manuscript.

\subsection*{Contribution Statement}

All authors contributed to the study conception and design. All authors
read and approved the manuscript.



\end{document}